\documentclass[11pt]{article}
\usepackage{epic,latexsym,amssymb}
\usepackage{color}
\usepackage{tikz}
\usepackage{amsmath,amsthm}
\usepackage{dirtytalk}
\usepackage{multirow}

\newtheorem{thm}{Theorem}
\newtheorem{lem}{Lemma}

\newtheorem{claim}{Claim}

\newtheorem{cor}{Corollary}

\newtheorem{defn}{Definition}

\newcommand{\Zekhaya}[1]{{\color{red} #1}}

\let\oldenumerate\enumerate
\renewcommand{\enumerate}{
  \oldenumerate
  \setlength{\itemsep}{0pt}
  \setlength{\parskip}{0pt}
  \setlength{\parsep}{0pt}
}

\def\vertex(#1){\put(#1){\circle*{2}}}
\def\vertexo(#1){\put(#1){\circle{2}}}
\def\vert(#1){\put(#1){\circle*{1.5}}}
\def\verto(#1){\put(#1){\circle{1.5}}}
\def\lab(#1)#2{\put(#1){\makebox(0,0)[c]{#2}}}

\begin{document}

\title{The number of $1$-nearly independent vertex subsets}

\author{Eric Ould Dadah Andriantiana \\
	Department of Mathematics (Pure and Applied) \\
	Rhodes University \\
	Makhanda, 6140 South Africa\\
	\small \tt Email: e.andriantiana@ru.ac.za \\\\Zekhaya B. Shozi\thanks{Research supported by the National Graduate Academy for Mathematical and Statistical Sciences (NGA-MaSS), Sol Plaatje University}\\
	Department of Mathematical Sciences\\
	Sol Plaatje University\\
	Kimberley, 8301 South Africa\\
\small \tt Email: zekhaya.shozi@spu.ac.za
}

\date{}
\maketitle

\begin{abstract}
 Let $G$  be a graph with vertex set $V(G)$ and edge set $E(G)$. A subset $I$ of $V(G)$ is an independent vertex subset if no two vertices in $I$ are adjacent in $G$. We study the number, $\sigma_1(G)$, of all subsets of  $v(G)$ that contain exactly one pair of adjacent vertices. We call those subsets 1-nearly independent vertex subsets. Recursive formulas of $\sigma_1$ are provided, as well as some cases of explicit formulas. We prove a tight lower (resp. upper) bound on $\sigma_1$ for graphs of order $n$. We deduce as a corollary that the star $K_{1,n-1}$ (the tree with degree sequence $(n-1,1,\dots,1)$) is the $n$-vertex tree with smallest $\sigma_1$, while it is well known that $K_{1,n-1}$ is the $n$-vertex tree with largest number of independent subsets.\end{abstract}

\newpage
	
\section{Introduction}

A simple and undirected graph $G$ is an ordered pair of sets $G=(V(G),E(G))$, where every element of the set of $E(G)$ of edges is a $2$-element subset of the set of vertices $V(G)$. $|V(G)|$ is the order of $G$, and $|E(G)|$ its size. For graph theory notation and terminology, we generally follow~\cite{henning2013total}. We use the standard notation $[k] = \{1,\ldots,k\}$.

An \emph{independent (vertex) subset} of a graph $G$, with vertex set $V(G)$ and edge set $E(G)$, is any subset of $V(G)$ that does not contain adjacent vertices. The number of independent  subsets of a graph is a popular graph invariant with rich literature. See the survey in \cite{wagner2010maxima}, where it is called the \emph{Merrifield-Simmons index}. See also the book \cite{li2012shi} for an even more extensive survey. The book itself is on the energy of graphs but given the strong connection between the number of independent vertex subsets and the energy of graphs, the book also contains a wealth of results on the number of independent vertex subsets.

Merrifield and Simmons used \cite{Merrifield198055} the number independent subsets  of molecular graphs and as  a measure of  molecular complexity, bond strength and boiling point  of the associated molecules. This boosted the interest of both chemists and mathematicians to study the invariant. Obtained results go beyond just the class of molecular graphs. In \cite{Andriantiana2013724} the structure of the tree with a given degree sequence $D$ that has the largest number of independent subsets is fully characterised. The result implies as corollaries characterisations of trees with largest number of independent subsets in various other classes like trees with fixed order, or with fixed order and given maximum degree.

Various ways of generalisation of the notion of independent subsets have been attempted. For example \cite{jagota2001generalization} generalised the concept of maximal-independent set, by considering the $k$-insulated set $S$ of a graph $G$ defined as a subset of its vertices such that each vertex in $S$ is adjacent to at most $k$ other vertices in $S$ and each vertex not in $S$ is adjacent to at least $k+1$ vertices in $S$. See also \cite{drmota1991generalized}, which studies subsets which do not contain pair of vertices with distance shorter than a specified integer $k$. In this paper, we propose a new other generalisation.

$H$ is a \emph{subgraph} of $G$ if $V(H)\subseteq V(G)$  and $E(H)\subseteq E(G)$. If furthermore 
$$
E(H)=E(G)\cap \{\{u,v\}:u,v\in V(G)\},
$$
then we say that $H$ is an \emph{induced subgraph} of $G$. This means that for a given set of vertices $V(H)$ it has all possible edges of $G$ that it can have. Note that to any independent subset $I$ of $G$ corresponds an induced subgraph $(I,\emptyset)$ of $G$. We propose one way to generalise the notion of independent subsets. We call any induced subgraph of size $k$ a $k$-nearly independent subset of $G$. Let $\sigma_k(G)$ denotes the number of $k$-nearly independent subsets of $G$. $\sigma_0(G)$ is the number of independent susbets of $G$. 

In this paper, we study $\sigma_1$. At least for the classes of graphs we investigated, the behaviour of $\sigma_1$ has little in common with that of $\sigma_0$. 
Among all graphs of order $n$, while the edgeless  graph $\overline{K_n}$ has the largest $\sigma_0$, it has the smallest $\sigma_1$ as $\sigma_1(\overline{K_n})=0$. The complete graph $K_n$, that has all possible edges, which has the smallest $\sigma_0$ does not reach an extremal value for $\sigma_1$ if $n\geq 6$: not the minimum nor the maximum. 

The rest of the paper is structured as follows. We start with a preliminary sections that contains technical formulas, some of which will be used in this paper, while others are included because we expect that they will be useful in further studies of $\sigma_1$. We also investigate the effect of an edge removal in that section and provide explicit expression of $\sigma_1$ for selected types of graphs. The main results are in Section \ref{Sec:Main}. There, we provide a full characterisation of the family of graphs that reach the minimum $\sigma_1$ among all graphs of order $n$ size $m$. Several corollaries follows from this. For example, if $m=n-1$, the star $K_{1,n-1}$ (the connected $n$-vertex graph with $n-1$ vertices of degree $1$) is a member of the family. Hence it has the minimum $\sigma_1$ among all trees, while it is known to have the maximum $\sigma_0$. The characterisation of the graph of order $n$ and maximum $\sigma_1$ that we characterise at the end of the section is a rather unusual graph. For $n\geq 8$ the maximum $\sigma_1$ is attained by two graph both with maximum degree $1$, one has size $3$ and the other has size $4$. 

\section{Preliminary}

 Let $G$ be a graph with vertex set $V(G)$, edge set $E(G)$, order $n = |V(G)|$ and size $m = |E(G)|$. We denote the degree of a vertex $v$ in $G$ by $\deg_Gv$. The maximum degrees in $G$ is denoted by $\Delta(G).$ 
 The complement of $G$ is denoted by $\overline{G}$ and defined as 
 $$\overline{G}=(V(G),\{uv: u,v\in V(G), u\neq v \text{ and }uv\notin E(G)\}).$$
 
 For a subset $S$ of vertices of a graph $G$, we denote by $G - S$ the graph obtained from $G$ by deleting the vertices in $S$ and all edges incident to them. If $S = \{v\}$, then we simply write $G - v$ rather than $G - \{v\}$. For any graphs $G$ and $H$, we define the \emph{join}
 $$
 G+H=(V(G)\cup V(H), E(G)\cup E(H)\cup\{uv: u\in V(G)\text{ and } v\in V(H) \})
 $$
 obtained by adding all possible edges between $G$ and $H$. 
 The subgraph of $G$ induced by the set $S$ is denoted by $\langle S \rangle_G$. We denote the path graph, cycle graph, wheel graph and complete graph on $n$ vertices by $P_n$, $C_n$, $W_n$ and $K_n$, respectively. 

For positive integers $r$ and $s$, we denote by $K_{r,s}$ the bipartite graph with partite sets $X$ and $Y$ such that $|X|=r$ and $|Y|=s$. A bipartite graph $K_{1,n-1}$ is also called a \emph{star}. Let $k\ge 2$ be an integer. A \emph{broom graph}, denoted $B_n^k$, is the tree of order $n$ obtained from the path, $P_k$, of order $k$ by adding $n-k$ new vertices and then joining them to exactly one end-vertex of $P_k$. A \emph{lollipop graph}, denoted $L_n^k$, is the graph of order $n$ obtained from a path $P_k$ of order $k$, by adding a complete graph, $K_{n-k}$, of order $n-k$ and then joining every vertex of $K_{n-k}$ to exactly one end-vertex of $P_k$. A \emph{tadpole}, denoted $T_n^k$, is the graph of order $n$ obtained from a path, $P_k$, of order $k$ by adding another path, $P_{n-k}$, of order $n-k$ and then joining the end-vertices of $P_{n-k}$ to exactly one end-vertex of $P_k$.

\subsection{Recursive formula}
Suppose that $G$ and $H$ are disjoint graphs and $K=G\cup H$. Then, every $1$-nearly independent subsets $I$ of $K$ splits uniquely into two parts: the part $I_G$ that is in $G$ and the part $I_H$ that is in $H$. If the edge in $I$ is from $G$ then $I_G$ is a $1$-nearly independent subset and $I_H$ an independent subset. If the edge in $I$ is from $H$, then $I_H$ is a $1$-nearly independent subset while $I_G$ is an independent subset. Hence we have
$$
\sigma_1(G\cup H)
=\sigma_1(G)\sigma_0(H)+\sigma_0(G)\sigma_1(H).
$$
For any vertex $v$ of a graph $G$
\begin{align}
\label{Eq:Rec}
\sigma_1(G)
=\sigma_1(G-v)+\sigma_1(G-N[v])+\sum_{u\in N(v)}\sigma_0(G-(N[u]\cup N[v]))>\sigma_1(G-v),
\end{align}
where $\sigma_1(G-v)$ counts the number of $1$-nearly independent subsets that do not contain $v$, $\sigma_1(G-N[v])$ counts those that contain $v$ as a single vertex, and $\sum_{\in N(v)}\sigma_0(G-(N[u]\cup N[v]))$ counts those that contain $v$ is as an edge. 

\subsection{Effect of an edge or a vertex removal on $\sigma_1$}

From Equation \ref{Eq:Rec}, we see that removing a vertex $v$ from a graph $G$ decreases its $\sigma_1$, or keeps it unchanged if $\sigma_1(G-N[v])=0$ and $N(v)=\emptyset$. This condition means that $v$ is an isolated vertex and $G-N[v]=G-v$ has no edge. The graph $G$ is edge-less in this case.

Removing an edge might leave the value of $\sigma_1$ unchanged, decreasing or increasing. Below are some examples of situations when some of these possibilities arises.

Note that for any edge $e$ of $K_3$ we have
$$
\sigma_1(K_3)=3>\sigma_1(K_3-e)=\sigma_1(P_3)=2,
$$
while for any edge $e'$ of $P_3$ we have
$$
\sigma_1(P_3)=2=\sigma_1(P_3-e')=2\sigma_1(K_2),
$$
and if $e''$ is the middle edge of $P_4$, then
$$
\sigma_1(P_4)=5<\sigma_1(P_4-e'')=\sigma_1(2K_2)=6.
$$


\subsection{Explicit formulas for $\sigma_1$ of some graphs}

Using known formulas for $\sigma_0$ and some of the recursive formulas in the previous subsection, we present in this section explicit formulas for $\sigma_1$ of a complete graph $K_n$, an unicyclic graph $U_n$, a path $P_n$, a cycle $C_n$,  and a wheel graph $W_n$ in terms of $n$. More formulas of the $\sigma_1$ of a broom graph $B_n^k$, a lollipop graph $L_n^k$, and a tadpole $T_n^k$ in terms of $n$, $k$ and $\sigma_1$ of paths are also presented. 

It is convenient to set $\sigma_1(P_t)=0$ whenever $t\leq 0$ and $\sigma_0(P_t)=1$ whenever $t\leq 0$.

Let $\alpha = \frac{1+\sqrt{5}}{2}$ and $\beta = \frac{1-\sqrt{5}}{2}$. Then, we have 
    $\alpha + \beta =1$,
     $\alpha - \beta = \sqrt{5}$,
     $\alpha\cdot \beta = -1$.

The following formulas are well known, see for example \cite{andriantiana2010number} and \cite{wagner2010maxima}. 

\begin{thm}[cf. \cite{andriantiana2010number}]
\label{thm:sigma-of-a-path-by-eric}
    For $n \in \mathbb{N}$, 
    we have
    \begin{align}
    \label{eq:sigma-of-a-path-by-eric}
        \sigma_0(P_n) = \frac{1}{\sqrt{5}} \left( \alpha^{n+2} - \beta^{n+2} \right).
    \end{align}
\end{thm}


\begin{thm}[cf. \cite{wagner2010maxima}]
    \label{thm:sigma-union-is-product-of-sigmas}
    If $G_1, G_2, \ldots, G_r$ are the connected components of a graph $G$, then
    \begin{align}
        \label{eq:sigma-union-is-product-of-sigmas}
        \sigma_0(G) = \sigma_0\left(\bigcup\limits_{i=1}^r G_i \right) = \displaystyle \prod\limits_{i=1}^r \sigma_0(G_i).
    \end{align}
\end{thm}

\subsubsection{$\sigma_1$ of $K_n$, $P_n$, $U_n$, $C_n$, and $W_n$}

Every edge of $K_n$ can only be contained in exactly one $1$-nearly independent subsets. Hence, we have
\begin{align*}
    \sigma_1(K_n) = |E(K_n)| = \binom{n}{2} = \frac{n(n-1)}{2}.
\end{align*}


Let $P_n = (\{v_1, \dots,v_n\},\{e_i=v_iv_{i+1}: 1\leq i\leq n-1\})$. 
Since the number of $1$-nearly independent subsets containing $e_i$ in $P_n$ is exactly $\sigma_0(P_{i-2} \cup P_{n-i-2})$, we have
\begin{align}
\label{eq:sum-equation-of-sigma-1-of-Pn}
\sigma_1(P_n) &= \displaystyle \sum\limits_{i=1}^{n-1} \sigma_0(P_{i-2} \cup P_{n-i-2}) \nonumber\\
&=\displaystyle \sum\limits_{i=1}^{n-1} \sigma_0(P_{i-2}) \sigma_0(P_{n-i-2}) & \text{ by Theorem \ref{thm:sigma-union-is-product-of-sigmas}}\nonumber\\
&= \frac{1}{5} \displaystyle \sum\limits_{i=1}^{n-1} (\alpha^i - \beta^i)(\alpha^{n-i} - \beta^{n-i}) & \text{ by Theorem \ref{thm:sigma-of-a-path-by-eric}} \nonumber\\
&= \frac{1}{5} \displaystyle \sum\limits_{i=1}^{n-1} \left[ \alpha^n - \beta^n \left(\frac{\alpha}{\beta}\right)^i - \alpha^n \left(\frac{\beta}{\alpha}  \right)^i + \beta^n \right] \nonumber \\
&=\frac{1}{5}(n-1)\left(\alpha^n + \beta^n \right) -\frac{1}{5} \left[ \beta^n \displaystyle \sum\limits_{i=1}^{n-1} \left(\frac{\alpha}{\beta}\right)^i  + \alpha^n \displaystyle \sum\limits_{i=1}^{n-1} \left(\frac{\beta}{\alpha}  \right)^i    \right].
\end{align}

The sum 
\begin{align*}
    \displaystyle \sum\limits_{i=1}^{n-1} \left(\frac{\alpha}{\beta}\right)^i = \frac{ \frac{\alpha}{\beta} \left[ \left( \frac{\alpha}{\beta} \right)^{n-1} -1   \right]   }{ \left( \frac{\alpha}{\beta} -1\right)},
\end{align*}
implies that 
\begin{align}
\label{eq:beta-n-sum-geometric-series}
     \beta^n \displaystyle \sum\limits_{i=1}^{n-1} \left(\frac{\alpha}{\beta}\right)^i = \alpha \beta \left( \frac{ \alpha^{n-1} - \beta^{n-1}  }{\alpha -\beta} \right) = - \frac{1}{\sqrt{5}}\left(  \alpha^{n-1} - \beta^{n-1}\right).
\end{align}

Similarly,
\begin{align*}
    \displaystyle \sum\limits_{i=1}^{n-1} \left(\frac{\beta}{\alpha}\right)^i = \frac{ \frac{\beta}{\alpha} \left[ \left( \frac{\beta}{\alpha} \right)^{n-1} -1   \right]   }{ \left( \frac{\beta}{\alpha} -1\right)},
\end{align*}
leads to
\begin{align}
\label{eq:alpha-n-sum-geometric-series}
     \alpha^n \displaystyle \sum\limits_{i=1}^{n-1} \left(\frac{\beta}{\alpha}\right)^i = \alpha \beta \left( \frac{ \alpha^{n-1} - \beta^{n-1}  }{\alpha -\beta} \right) = - \frac{1}{\sqrt{5}}\left(  \alpha^{n-1} - \beta^{n-1}\right).
\end{align}
Thus, substituting Equations \eqref{eq:beta-n-sum-geometric-series} and \eqref{eq:alpha-n-sum-geometric-series} into Equation \eqref{eq:sum-equation-of-sigma-1-of-Pn} yields
\begin{align*}
    \sigma_1(P_n) &= \frac{1}{5}(n-1)\left(\alpha^n + \beta^n \right) -\frac{1}{5} \left[- \frac{1}{\sqrt{5}}\left(  \alpha^{n-1} - \beta^{n-1}\right) - \frac{1}{\sqrt{5}}\left(  \alpha^{n-1} - \beta^{n-1}\right)    \right]\\
   &= \frac{1}{5}\left[(n-1)\left(\alpha^n + \beta^n \right) + \frac{2}{\sqrt{5}} \left( \alpha^{n-1} -\beta^{n-1} \right) \right].
\end{align*}


Let $U_n$ be the graph obtained from the tree $K_{1,n-1}$ by adding one more edge to make it a unicyclic graph. Then, clearly, $\Delta(U_n) = n-1$. Let $v$ be the vertex of maximum degree in $U_n$. Then, we have
\begin{align*}
\sigma_1(U_n)&=\sigma_1(U_n-v) + \sigma_1(U_n - N_G[v]) + \displaystyle \sum\limits_{u\in N_G(v)}\sigma_0(U_n-(N_G[u]\cup N_G[v]))\\
&=\sigma_1(P_2\cup \overline{K_{n-3}}) + \sigma_1(\emptyset) + (n-1)\sigma_0(\emptyset)\\
&=\sigma_1(P_2)\sigma_0(\overline{K_{n-3}}) + \sigma_1(\overline{K_{n-3}})\sigma_0(P_2)+\sigma_1(\emptyset) + (n-1)\sigma_0(\emptyset)\\
&=n-1+2^{n-3}.
\end{align*}

Let $C_n=(\{v_1,\dots,v_n\},\{e_i=v_iv_{i+1}:1\leq i\leq n-1\}\cup \{e_n=v_nv_1\})$. Note that if we remove consecutive vertices of $C_n$, we obtain a path of order less than $n$. 
Every edge of $C_n$ is contained in exactly $\sigma_0(P_{n-4})$ $1$-nearly independent subsets. Thus,
\begin{align*}
    \sigma_1(C_n) &= \displaystyle \sum\limits_{i=1}^n \sigma_0(P_{n-4})=n\sigma_0(P_{n-4}) = \frac{n}{\sqrt{5}}\left( \alpha^{n-2} - \beta^{n-2} \right).
\end{align*}


A \emph{wheel graph} $W_n$, of order $n$ is given by $W_n = C_{n-1} + K_1$. Note that in $W_n$, a vertex $v$ is central if and only if $\deg_{W_n}v = n-1$. Let $v$ be a central vertex of $W_n$. Then
\begin{align*}
        \sigma_1(W_n) &= \sigma_1(W_n-v) + \sigma_1(W_n - N_G[v]) + \displaystyle \sum\limits_{u\in N_G(v)}\sigma_0(W_n-(N_G[u]\cup N_G[v]))\\
        &=\sigma_1(C_{n-1}) + \sigma_1(\emptyset) + (n-1)\sigma_0(\emptyset)\\
        &=\frac{1}{\sqrt{5}}(n-1)\left(\alpha^{n-3} -\beta^{n-3} \right) + (n-1)\\
        &=(n-1)\left[1 + \frac{1}{\sqrt{5}} \left( \alpha^{n-3} -\beta^{n-3} \right)  \right].
\end{align*}

\subsubsection{$\sigma_1$ of $B_n^k$, $L_n^k$, and $T_n^k$}

Let $v$ be a vertex of $B_n^k$ of maximum degree. Then $\deg_{B_n^k} v = n-k+1$, and so
\begin{align*}
    \sigma_1(B_n^k) = &\sigma_1(B_n^k-v) + \sigma_1(B_n^k - N_G[v]) + \displaystyle \sum\limits_{u\in N_G(v)}\sigma_0(B_n^k-(N_G[u]\cup N_G[v]))\\
    =&\sigma_1(P_{k-1} \cup \overline{K_{n-k}}) + \sigma_1(P_{k-2}) + \sigma_0(P_{k-3}) + (n-k)\sigma_0(P_{k-2})\\
    =&\sigma_1(P_{k-1})\sigma_0(\overline{K_{n-k}}) + \sigma_1(\overline{K_{n-k}})\sigma_0(P_{k-1}) + \sigma_1(P_{k-2}) + \sigma_0(P_{k-3}) \\
    &+ (n-k)\sigma_0(P_{k-2})\\
    =&\sigma_1(P_{k-1})\cdot 2^{n-k} + \sigma_1(P_{k-2}) + \sigma_0(P_{k-3}) + (n-k)\sigma_0(P_{k-2}).
\end{align*}


Let $v$ be a vertex of $L_n^k$ of maximum degree. Then $\deg_{L_n^k} v = n-k+1$, and so
\begin{align*}
	  \sigma_1(L_n^k) = &\sigma_1(L_n^k-v) + \sigma_1(L_n^k - N_G[v]) + \displaystyle \sum\limits_{u\in N_G(v)}\sigma_0(L_n^k-(N_G[u]\cup N_G[v]))\\
	  =&\sigma_1(P_{k-1}\cup K_{n-k}) + \sigma_1(P_{k-2}) + \sigma_0(P_{k-3}) + (n-k)\sigma_0(P_{k-2})\\
	  =&\sigma_1(P_{k-1})\sigma_0(K_{n-k}) + \sigma_1(K_{n-k})\sigma_0(P_{k-1}) + \sigma_1(P_{k-2}) + \sigma_0(P_{k-3}) \\
	  &+ (n-k)\sigma_0(P_{k-2})\\
	  =&(n-k+1)\sigma_1(P_{k-1}) + \binom{n-k}{2} \sigma_0(P_{k-1}) + \sigma_1(P_{k-2}) + \sigma_0(P_{k-3})\\ &+(n-k)\sigma_0(P_{k-2}).
\end{align*}


Let $v$ be a vertex of $T_n^k$ of maximum degree. Then $\deg_{T_n^k} v = 3$, and so
\begin{align*}
	\sigma_1(T_n^k) = &\sigma_1(T_n^k-v) + \sigma_1(T_n^k - N_G[v]) + \displaystyle \sum\limits_{u\in N_G(v)}\sigma_0(T_n^k-(N_G[u]\cup N_G[v]))\\
	=&\sigma_1(P_{k-1}\cup P_{n-k}) + \sigma_1(P_{k-2}\cup P_{n-k-2}) + \sigma_0(P_{k-3}\cup P_{n-k-2})\\
	&+ 2\sigma_0(P_{k-2}\cup P_{n-k-3})\\
	=&\sigma_1(P_{k-1})\sigma_0(P_{n-k}) + \sigma_1(P_{n-k})\sigma_0(P_{k-1}) + 
	\sigma_1(P_{k-2})\sigma_0(P_{n-k-2}) \\
	&+ \sigma_1(P_{n-k-2})\sigma_0(P_{k-2}) + \sigma_0(P_{k-3})\sigma_0(P_{n-k-2}) + 2\sigma_0(P_{k-2})\sigma_0(P_{n-k-3}).
\end{align*}

\section{Main result}
\label{Sec:Main}
In this section we present some bounds on the number of $1$-nearly vertex independent subsets of a graph $G$. In particular, we present a lower bound on the number of $1$-nearly vertex independent subsets of a connected graph $G$ with $m$ edges. 
Thereafter, we present an upper bound on the number of $1$-nearly vertex independent subsets of a graph $G$ that is not necessarily connected. The families of graphs that achieve equality on each of these bounds are characterised. 


\subsection{Minimal connected graphs}
In this subsection we characterise the family of connected graphs with $m$ edges and minimum number of $1$-nearly independent vertex subsets. This will be used to determine the connected graph, and then the tree, with given order and minimum number of $1$-nearly independent vertex subsets.

\begin{defn}
\label{Def:Good}
    Let $G = (V(G), E(G))$ be a graph with vertex set $V(G)$ and edge set $E(G)$. If $e=uv \in E(G)$, then $e$ is a good edge if $N_G[u]\cup N_G[v] = V(G)$. The graph $G$ is a good graph if for every edge $e \in E(G)$, $e$ is a good edge. Let
    $$\mathcal{H} = \{ G \mid G \text{ is a good graph} \}.$$
\end{defn}
It follows from the definition that a good graph has to be connected.

\begin{thm}
    If $G$ is a connected graph of size $m$, then
    \begin{align}
        \label{if-G-is-a-connected-graph-of-order-n-size-m-then-sigma1G-atleast-m}
        \sigma_1(G) \ge m,
    \end{align}
    with equality if and only if $G \in \mathcal{H}$.
\end{thm}

\begin{proof}
Inequality (\ref{if-G-is-a-connected-graph-of-order-n-size-m-then-sigma1G-atleast-m}) follows from the fact that the two ends of any edge of a graph $G$ form a $1$-nearly independent vertex subset of $G$. If $G$ is an $(n,m)$-graph that is not good, then there is an edge $uv$ of $G$ that is not a good edge. That is, $V(G)\setminus(N_G[u]\cup N_G[v])\neq \emptyset$. In addition to the $1$-nearly independent vertex subsets made of the two ends of the edge $uv$, the subset $\{u,v,z\}$ is also a $1$-nearly independent vertex subset, for any $z\in V(G)\setminus (N_G[u]\cup N_G[v]).$ Hence, we have
$\sigma_1(G)>m$. 
\end{proof}

Definition \ref{Def:Good} already gives a full characterisation of the family $\mathcal{H}$. However, using that description, it is still hard to imagine how the structure of an element of $\mathcal{H}$ with a large number of vertices should look like. The rest of this section is a further investigation of $\mathcal{H}$ aiming to understand the structures of its elements.


\begin{lem}
\label{Lem:ClosedH}
    The family $\mathcal{H}$ is closed under the join operation.
\end{lem}

\begin{proof}
    Let $G_1 = (V(G_1), E(G_1))$ and $G_2 = (V(G_2), E(G_2))$ be two graphs in $\mathcal{H}$. Then $G_1$ and $G_2$ are good graphs.  We want to show that $G_1+G_2\in \mathcal{H}$. Since $G_1$ and $G_2$ are good graphs, we have $N_{G_1}[s] \cup N_{G_1}[t] = V(G_1)$ for all $st\in E(G_1)$, and $N_{G_2}[u] \cup N_{G_2}[v] = V(G_2)$ for all $uv\in E(G_2)$. Let $wx \in E(G_1+G_2)$. 
    Then either $wx \in E(G_1)$ or $wx \in E(G_2)$ or $w \in V(G_1)$ and $x\in V(G_2)$. If $wx \in E(G_1)$, then, since $G_1$ is a good graph, $wx$ is a good edge in $G_1$. By the definition of the join operation, $N_{G_1+G_2}[w] = N_{G_1}[w] \cup V(G_2)$ and $N_{G_1+G_2}[x] = N_{G_1}[x] \cup V(G_2)$. Thus, $$N_{G_1+G_2}[w] \cup N_{G_1+G_2}[x] = N_{G_1}[w] \cup N_{G_1}[x] \cup V(G_2) = V(G_1)\cup V(G_2) = V(G_1+G_2).$$ By the same reasoning, if $wx \in E(G_2)$, then, since $G_2$ is a good graph, $wx$ is a good edge in $G_2$, and by the definition of the join operation, $wx$ is also a good edge in $G_1+G_2$.  Hence, we assume that $w \in V(G_1)$ and $x \in V(G_2)$. By the definition of the join operation, $N_{G_1+G_2}[w] = N_{G_1}[w] \cup V(G_2)$ and $N_{G_1+G_2}[x] = N_{G_2}[x] \cup V(G_1)$. Since $N_{G_1}[w] \subseteq V(G_1)$ and $N_{G_2}[x] \subseteq V(G_2)$, we have $$N_{G_1+G_2}[w] \cup N_{G_1+G_2}[x] = V(G_1)\cup V(G_2)=V(G_1+G_2),$$ implying that $wx$ is also a good edge in $G_1+G_2$. Since $wx$ was arbitrarily chosen, the graph $G_1+G_2$ is a good graph,  and thus $G_1 + G_2 \in \mathcal{H}$. 
\end{proof}

\begin{lem}
\label{Lem:ExtendH}
    If $G \in \mathcal{H}$, then for any integer $\ell \ge 1$, $G + \overline{K_\ell} \in \mathcal{H}$, where $\overline{K_\ell}$ is an edgeless graph of order $\ell$.
\end{lem}

\begin{proof}
    Let $G \in \mathcal{H}$. Then $G$ is a good graph, and so $N_G[u] \cup N_G[v] = V(G)$ for all $uv \in E(G)$. Let $\overline{K_\ell}$ be the empty graph of order $\ell$ with vertex set $V(\overline{K_\ell}) = \{w_1, w_2, \ldots, w_\ell\}$. By the join operation, $N_{G+\overline{K_\ell}} [w_i] = V(G)\cup \{w_i\}$ for all $i \in [\ell]$. Let $vx \in E(G+\overline{K_\ell})$. 
    Since $\overline{K_\ell}$ is an edgeless graph, $vx \notin E(\overline{K_\ell})$. If $vx \in E(G)$, then, since $G$ is a good graph, $vx$ is a good edge in $G$. By the definition of the join operation, $N_{G+\overline{K_\ell}}[v] = N_{G}[v] \cup V(\overline{K_\ell})$ and $N_{G+\overline{K_\ell}}[x] = N_{G}[x] \cup V(\overline{K_\ell})$. Thus, $$N_{G+\overline{K_\ell}}[v] \cup N_{G+\overline{K_\ell}}[x] =  N_{G}[v] \cup N_{G}[x] \cup V(\overline{K_\ell}) = V(G) \cup V(\overline{K_\ell}) = V(G + \overline{K_\ell}),$$ implying that $vx$ is also a good edge in $G+\overline{K_\ell}$. Hence, we assume that $v\in V(G)$ and $x\in V(\overline{K_\ell})$. Thus, $x=w_i$ for some $i\in [\ell]$. By the definition of the join operation, $N_{G+\overline{K_\ell}} [v] = N_G[v] \cup V(\overline{K_\ell})$. Since $N_G[v] \subseteq V(G)$, we have $$N_{G+\overline{K_\ell}} [v] \cup N_{G+\overline{K_\ell}} [x] = N_{G+\overline{K_\ell}} [v] \cup N_{G+\overline{K_\ell}} [w_i] = V(G)  \cup V(\overline{K_\ell})= V(G+\overline{K_\ell}),$$ implying that $vx = vw_i$ is a good edge in $G+\overline{K_\ell}$. Since $wx$ was arbitrarily chosen, the graph $G+\overline{K_\ell}$ is a good graph, and thus $G+\overline{K_\ell} \in \mathcal{H}$.
\end{proof}

Let $\mathcal{H}_1=\{K_1\}\cup \{K_{r,s}\mid r,s\in \mathbb{N}\}$. For any integer $k\geq 2$, we define
$$\mathcal{H}_k=\{K+H\mid K,H\in \mathcal{H}_{k-1}\}\cup \{G+\overline{K_\ell}\mid G\in \mathcal{H}_{k-1} \text{ and }\ell\in\mathbb{N} \}.$$

\begin{thm}
$$\mathcal{H}=\bigcup_{k\in\mathbb{N}}\mathcal{H}_k.$$
\end{thm}

\begin{proof}
It is easy to check that $\mathcal{H}_1\subseteq \mathcal{H}$. Also, by Lemma \ref{Lem:ClosedH} and Lemma \ref{Lem:ExtendH}, it follows that $\bigcup_{k\in\mathbb{N}}\mathcal{H}_k\subseteq \mathcal{H}$. Thus, it remains for us to prove the reverse inclusion. Suppose that $G$ is a good graph. Then $G$ is connected. If $G$ has only one vertex, then $G \in \mathcal{H}_1 \subseteq \mathcal{H}$. Hence, we may assume that $G$ has more than one vertex. 
Thus $G$ has at least one edge. We now consider the subgraph $H$ of $G$ with largest number of vertices, and such that $H=K+L$ for some induced subgraphs $K$ and $L$ of $G$. Suppose that $H$ is not $G$. Since $G$ is connected, there is a vertex $u$ in $G-H$ that is adjacent to a vertex $v$ in $H$. Without loss of generality, we can assume that $v\in V(K)$. Then, the edge $uv$ has to be a good edge; that is, $V(H) = V(K)\cup V(L)\subseteq N_G[u]\cup N_G[v]=V(G)$.

If there is a vertex $w\in V(K)$ that is not adjacent to $u$, then every vertex in $L$ has to be adjacent to $u$ (because every edge $ws$ for any $s\in V(L)$ has to be a good edge). In this case, we could add $u$ to $K$ and obtain a bigger subgraph of $G$, namely $H'=\langle V(K)\cup\{u\}\rangle_G+L$. However, this contradicts our choice of $H$. Hence, we may assume that all the vertices in $K$ are adjacent to $u$. In this case, we can again add $u$ to $L$ and have a bigger subgraph of $G$, namely $H''=K+\langle V(L)\cup\{u\}\rangle_G$. Once again, this is a contradiction to our choice of $H$. Hence, we must have $V(H)=V(G)$.

It is only left to prove that each of $K$ and $L$ is either an edgeless graph or a good graph.
It is sufficient to prove that if $K$ (or $L$) has an edge then it is a good graph. Suppose that $xy$ is an edge in $K$. Then it has to be a good edge in $G$; that is, $N_G[x]\cup N_G[y]=V(G)$. Thus, $N_K[x]\cup N_K[y]= (N_G[x]\cup N_G[y])\cap V(K)=V(G)\cap V(K)=V(K)$, implying that $xy$  also a good edge in $K$. 
\end{proof}

Since the star $K_{1,n-1}$ is the element of $\mathcal{H}$ with fewest edges among all elements of $\mathcal{H}$ with $n$ vertices, and it is a tree, we obtain the following corollaries:

\begin{cor}
    \label{among-all-connected-graphs-of-order-n-the-star-has-the-smallest-sigma-1}
    If $G$ is a connected graph of order $n$, then 
    \begin{align}
        \sigma_1(G) \ge n-1,
    \end{align}
    with equality if and only if $G \cong K_{1, n-1}$.
\end{cor}
\begin{cor}
    \label{Cor:Tree}
    Among all trees $T$ of order $n$, we have
    \begin{align}
        \sigma_1(T) \ge n-1,
    \end{align}
    with equality if and only if $G \cong K_{1, n-1}$.
\end{cor}
Note that $K_{1,n-1}$ has the largest value if one counts the number of independent vertex subsets instead.


\subsection{Maximal graphs}

In this subsection we characterise the family of graphs with $n$ vertices and maximum number of $1$-nearly independent vertex subsets. 

The edgeless graph $\overline{K_n}$, that has the maximum $\sigma_0$ among all graphs of order $n$, has the minimum $\sigma_1$ as $\sigma_1(\overline{K_n})=0$. We observe that for $n\geq 7$, $\sigma_1(U_n) =2^{n-3}+n-1> n(n-1)/2=\sigma_1(K_n)$. So, $K_n$ is neither be a minimal graph nor a maximal graph in this case.

\begin{thm}
\label{if-G-is-a-graph-of-order-n-6-then-sigma-1-at-most-max}
    If $G$ is a graph of order $n \ge 6$, then
    \begin{align*}
        \sigma_1(G) \le \frac{27}{64}\cdot 2^n,
    \end{align*}
    with equality if and only if $G \cong 3K_2 \cup (n-6)K_1$ or $G \cong 4K_2 \cup (n-8)K_1$.
\end{thm}

\begin{proof}
Note that for $n\geq 8$, we have
$$
\sigma_1(3K_2\cup (n-6)K_1)=3\cdot 3^22^{n-6}=\frac{27}{64}\cdot 2^n=4\cdot 3^32^{n-8}=\sigma_1(4K_2\cup (n-8)K_1).
$$
    We use induction on the order $n$ of a graph $G$. If $n= 6$, then, by Table \ref{sigma-1-in-graphs-of-order-6}, $\sigma_1(G) \le 27 =  \frac{27}{64}\cdot 2^n$, where $3K_2$ is the only one that reaches $27$. this proves the base case. Now we may assume that the result holds for all graphs $G'$ of order $n<k$, where $k \ge 7$. Let $G$ be a graph of order $n=k$ and maximum $\sigma_1$.  We proceed with the following series of claims.

\begin{claim}
\label{if-G-contains-an-isolated-vertex-then-sigma-1-atmost-max}
    If $G$ contains an isolated vertex, then $\sigma_1(G) \le \frac{27}{64}\cdot 2^n$, with equality only if $G\in \{3K_2\cup (n-6)K_1,4K_2\cup (n-8)K_1\}$.
\end{claim}

\begin{proof}
    Suppose that $G$ contains an isolated vertex $v$. Then 
\begin{align*}
    \sigma_1(G) &= 2\sigma_1(G-v)\le 2 \left(\frac{27}{64}\cdot 2^{n-1}\right)= \frac{27}{64}\cdot 2^n.
\end{align*}
 The equality only holds if $G-v\in\{3K_2\cup ((n-1)-6)K_1),\sigma_1(4K_2\cup ((n-1)-8)K_1)\}$, in which case $G\in\{3K_2\cup (n-6)K_1),\sigma_1(4K_2\cup (n-8)K_1)\}$.
 This completes the proof of Claim \ref{if-G-contains-an-isolated-vertex-then-sigma-1-atmost-max}.
\end{proof}

Thus, by Claim \ref{if-G-contains-an-isolated-vertex-then-sigma-1-atmost-max}, we may assume that $G$ does not contain an isolated vertex.

\begin{claim}
\label{if-max-degree-one-then-sigma-1-is-atmost-max}
    If $\Delta(G) = 1$, then $\sigma_1(G) \le  \frac{27}{64}\cdot 2^n$, where equality only holds if $n\in\{6,8\}$.
\end{claim}

\begin{proof}
    Suppose $\Delta(G) = 1$. Then, since $G$ does not contain any isolated vertex, $G \cong \left(\frac{n}{2} \right)K_2$, and we have $\sigma_1(G) = \left(\frac{n}{2}\right) (3)^{\frac{n}{2} -1} = \left(\frac{n}{6}\right) (\sqrt{3})^n$. Consider the function 
\begin{align*}
    f(n) = \frac{1}{2^n} \left( \sigma_1(G) -  \frac{27}{64}\cdot 2^n \right) = \left(\frac{n}{6}\right) \left(\frac{\sqrt{3}}{2}\right)^n - \frac{27}{64}.
\end{align*}
Then
\begin{align*}
    f'(n) = \frac{1}{6}\left(\frac{\sqrt{3}}{2}\right)^n + \left(\frac{n}{6}\right) \left(\frac{\sqrt{3}}{2}\right)^n \ln\left(\frac{\sqrt{3}}{2}\right) = \frac{1}{6} \left( 1 + n\ln\left(\frac{\sqrt{3}}{2}\right) \right) \left(\frac{\sqrt{3}}{2}\right)^n.
\end{align*}
Since $\ln\left(\frac{\sqrt{3}}{2}\right) < -0.14$, we have $n\ln\left(\frac{\sqrt{3}}{2}\right) < -1$ for all $n\ge 10$, and so $f'(n) < 0$ for all $n\ge 10$. Therefore, for all $n\ge 10$, we have $f(n) < f(10) < -0.1 \le 0$. We deduce, therefore, that for $n\ge 10$, if $G$ is a graph of order $n$ with no isolated vertices, and with $\Delta(G) =1$, then   $\sigma_1(G) <  \frac{27}{64}\cdot 2^n$.
\end{proof}

By Claim \ref{if-G-contains-an-isolated-vertex-then-sigma-1-atmost-max} and Claim \ref{if-max-degree-one-then-sigma-1-is-atmost-max}, we may assume that $G$ does not contain an isolated vertex and $\Delta(G) \ge 2$. Let $v$ be a vertex of $G$ with maximum degree $q\ge 2$.

\begin{claim}
    \label{if-q-at-least-4-then-sigma-1-at-most-max}
    If $q \ge 4$, then  $\sigma_1(G) <  \frac{27}{64}\cdot 2^n$.
\end{claim}

\begin{proof}
    Suppose $q\ge 4$, and let $v$ be a vertex of $G$ with maximum degree $q$. Thus,
\begin{align*}
    \sigma_1(G) &= \sigma_1(G-v) + \sigma_1(G - N_G[v]) + \displaystyle \sum\limits_{u\in N_G(v)}\sigma_0(G-(N_G[u]\cup N_G[v]))\\
    &\le \frac{27}{64}\cdot 2^{n-1} + \frac{27}{64}\cdot 2^{n-1-q} + q\cdot 2^{n-1-q}\\
    &=\frac{27}{64}\cdot 2^n \left( \frac{1}{2} + \frac{1}{2^{q+1}} \left( 1 + \frac{64}{27}q \right)  \right).
\end{align*}
Now it remains for us to show that for all $q \in \mathbb{Z}$ such that $q \ge 4$, $\frac{1}{2^{q+1}} \left( 1 + \frac{64}{27}q \right) < \frac{1}{2}$. We use induction on $q\ge 4$. If $q=4$, then $\frac{1}{2^{4+1}} \left( 1+ \frac{64}{27}(4) \right) = 0.327 < \frac{1}{2}$, thereby proving the base case. Assume that the result holds for $q = k \ge 4$. That is, assume that $\frac{1}{2^{k+1}} \left( 1 + \frac{64}{27}k \right) < \frac{1}{2}$. Thus, if $q=k+1$, we have
\begin{align*}
    \frac{1}{2^{k+1+1}} \left( 1 + \frac{64}{27}(k+1) \right) &= \frac{1}{2} \cdot \frac{1}{2^{k+1}} \left( 1 + \frac{64}{27}k + \frac{64}{27} \right)\\
    &= \frac{1}{2} \cdot \frac{1}{2^{k+1}} \left( 1 + \frac{64}{27}k\right)  +  \frac{64}{54} \cdot \frac{1}{2^{k+1}}.\\
    &\le \frac{1}{2}\cdot \frac{1}{2} + \frac{1}{27} & \text{since } k \ge 4\\
    &=\frac{31}{108}< \frac{1}{2}.
\end{align*}
Thus, by mathematical induction, $\frac{1}{2^{q+1}} \left( 1 + \frac{64}{27}q \right) < \frac{1}{2}$ for all $q \in \mathbb{Z}$ such that $q \ge 4$.

Since $\frac{1}{2^{q+1}} \left( 1 + \frac{64}{27}q \right) < \frac{1}{2}$ for all $q \in \mathbb{Z}$ such that $q \ge 4$, we deduce, therefore, that if $\Delta(G) = q \ge 4$, then  $\sigma_1(G) <  \frac{27}{64}\cdot 2^n$.
\end{proof}

By Claim \ref{if-q-at-least-4-then-sigma-1-at-most-max} (as well as Claim \ref{if-G-contains-an-isolated-vertex-then-sigma-1-atmost-max} and Claim \ref{if-max-degree-one-then-sigma-1-is-atmost-max}), we may assume that $2\le q \le 3$.

\begin{claim}
\label{if-q-2-or-q-3-and-G-connected-then-sigma1-max}
    If $G$ is a connected graph with $\Delta(G)=q$, where $2\le q \le 3$, then $\sigma_1(G) < \frac{27}{64}\cdot 2^n$.
\end{claim}

\begin{proof}
    Let $G$ be a connected graph of order $n\ge 7$ with $\Delta(G)=q$, where $2\le q\le 3$. If $q=2$, then $G \cong P_n$ or $G \cong C_n$. Let $v$ be a central vertex in $G$. Then $\deg_Gv =2$, and we have
\begin{align*}
    \sigma_1(G) &= \sigma_1(G-v) + \sigma_1(G - N_G[v]) + \displaystyle \sum\limits_{u\in N_G(v)}\sigma_0(G-(N_G[u]\cup N_G[v]))\\
    &\le \frac{27}{64}\cdot 2^{n-1} + \frac{27}{64}\cdot 2^{n-3} + 2(2^{n-4})\\
    &\le \frac{27}{64}\cdot 2^{n}\left( \frac{1}{2} + \frac{1}{8} + \frac{64}{27}\cdot \frac{1}{8}  \right) <\frac{27}{64}\cdot 2^n.
\end{align*}
Hence, we may assume that $q=3$. Let $v$ be a vertex of degree $3$ in $G$.  Since $n\ge 7$, there must exist at least one neighbour of $v$ having degree at least $2$. Thus, we have
\begin{align*}
    \sigma_1(G) &= \sigma_1(G-v) + \sigma_1(G - N_G[v]) + \displaystyle \sum\limits_{u\in N_G(v)}\sigma_0(G-(N_G[u]\cup N_G[v]))\\
    &\le \frac{27}{64}\cdot 2^{n-1} + \frac{27}{64}\cdot 2^{n-4} + 2(2^{n-4}) + 2^{n-5}\\
    &\le \frac{27}{64}\cdot 2^{n}\left( \frac{1}{2} + \frac{1}{16} + \frac{64}{27}\cdot \frac{1}{8} + \frac{64}{27}\cdot \frac{1}{32} \right) <\frac{27}{64}\cdot 2^n.
\end{align*}
This completes the proof of Claim \ref{if-q-2-or-q-3-and-G-connected-then-sigma1-max}.
\end{proof}

Hence, by Claim \ref{if-q-2-or-q-3-and-G-connected-then-sigma1-max}, we may assume that $2 \le q \le 3$ and $G$ has at least two components.

\begin{claim}
\label{if-G'-is-a-component-then-it-must-be-3K_2-free}
    If $G'$ is a component of $G$, then neither $G'$ nor $G-G'$ contains $3K_2$, otherwise $\sigma_1(G)$ does not reach the maximum value.
\end{claim}

\begin{proof}
    Suppose that $G'$ is a component of $G$ with $t$ vertices, and one of $G'$ and $G-G'$, say $G'$, contains $3K_2$. Remove edges from $G'$ to make it a $3K_3\cup (|V(G)|-6)K_1$. By the inductive hypothesis (and the the fact that removing any edge increases $\sigma_0$), we have $\sigma_1(G')\leq \sigma_1(3K_3\cup (t-6)K_1)$ and $\sigma_0(G')<\sigma_0(3K_3\cup (t-6)K_1)$. Thus,
    \begin{align*}
    \sigma_1(G) &=\sigma_1(G')\sigma_0(G-G')+\sigma_0(G')\sigma_1(G-G')\\
    &<\sigma_1(3K_3\cup (t-6)K_1)\sigma_0(G-G')+\sigma_0(3K_3\cup (t-6)K_1)\sigma_1(G-G')=\sigma_1(G''),
    \end{align*}
    where $G''$ is obtained from $G$ by replacing $G'$ by $3K_3\cup (t-6)K_1$. 
\end{proof}

Therefore, Claim \ref{if-G'-is-a-component-then-it-must-be-3K_2-free} (combined with the fact that $G$ cannot contain an isolated vertex) implies that $G$ has at most three connected components, otherwise three of the component would contain $3K_2$).

Suppose that $G$ has three components, namely $G_1$, $G_2$ and $G_3$. Then, in each of the components, any pair of edges are adjacent (otherwise two of the components would have a $3K_2$). This implies that, for $i\in [3]$, $G_i$ is a star or a complete graph $K_3$. Let $G_1$ be the largest component of $G$. Then $G_1$ has at least $3$ vertices, since the order of $G$ is at least $7$.
If $G_1$ is a star, let $v$ be a vertex of degree $1$ in $G_1$ attached at a vertex $u$ of degree $q$, where $2\le q\le 3$. Then $G-G_1$ contains $2K_2$ and thus 
\begin{align*}
\sigma_1(G) &= \sigma_1(G-v) + \sigma_1(G - N_G[v]) + \displaystyle \sum\limits_{u\in N_G(v)}\sigma_0(G-(N_G[u]\cup N_G[v]))\\
&=\sigma_1(G-v)+\sigma_1(G-v-u)+\sigma_0(G-v-N[u])\\
&\le \frac{27}{64}\cdot 2^{n-1} + \frac{27}{64}\cdot 2^{n-2} + 3^2\cdot 2^{n-1-q-4}\\
&=\frac{27}{64}\cdot 2^n \left(\frac{1}{2} + \frac{1}{4} + \frac{64}{27}\cdot 9 \cdot \frac{1}{2^{5+q}} \right)<\frac{27}{64}\cdot 2^n, & \text{ since } q \ge 2.
\end{align*}
So, $G$ would not have maximum $\sigma_1$ in this case. 

If $G_1$ is a complete graph $K_3$ with vertices $v,u$ and $z$, then we 
have
\begin{align*}
\sigma_1(G) &= \sigma_1(G-v) + \sigma_1(G - N_G[v]) + \displaystyle \sum\limits_{u\in N_G(v)}\sigma_0(G-(N_G[u]\cup N_G[v]))\\
&=\sigma_1(G-v)+\sigma_1(G-v-u-z)+2\sigma_0(G-v-u-z)\\
&\le \frac{27}{64}\cdot 2^{n-1} + \frac{27}{64}\cdot 2^{n-3} + 2\cdot 3^2 \cdot 2^{n-1-2-4}\\
&=\frac{27}{64}\cdot 2^n \left( \frac{1}{2} + \frac{1}{8} + \frac{1}{3}\right)<\frac{27}{64}\cdot 2^n.
\end{align*}
Again, $G$  does not have the maximum $\sigma_1$.

Suppose that $G$ has two components, namely $G_1$ and $G_2$. Then, neither of them is a $K_1$ nor contains $3K_2$. 
Suppose that $G_1$ has a vertex with degree $\Delta(G)$. Suppose that $\Delta(G) = q =3$. 
Since $G-G_1$ contains $K_2$, we have 
\begin{align*}
\sigma_1(G)
\sigma_1(G) &= \sigma_1(G-v) + \sigma_1(G - N_G[v]) + \displaystyle \sum\limits_{u\in N_G(v)}\sigma_0(G-(N_G[u]\cup N_G[v]))\\
&\le \frac{27}{64}\cdot 2^{n-1} + \frac{27}{64}\cdot 2^{n-1-3} + 3\cdot 3 \cdot 2^{n-1-3-2}\\
&= \frac{27}{64}\cdot 2^n \left( \frac{1}{2} + \frac{1}{16} + \frac{1}{3} \right)<\frac{27}{64}\cdot 2^n.
\end{align*}
Again, $G$  does not have the maximum $\sigma_1$.

Hence, we may assume that $\Delta(G) = q = 2$. Then, each component $G_i$, of $G$ (for $i\in [2]$) is a path or a cycle. Since $G_i$ cannot contain $3K_2$, we deduce, therefore, that the order of $G_i$ is at most $5$, and hence the order of $G = G_1\cup G_2$ is at most $10$. Suppose one component, say $G_1$, is a path, and let $v$ be an end-vertex of $G_1$. Then, $G-G_1$ contains $K_2$, and we have 
\begin{align*}
    \sigma_1(G) &= \sigma_1(G-v) + \sigma_1(G - N_G[v]) + \displaystyle \sum\limits_{u\in N_G(v)}\sigma_0(G-(N_G[u]\cup N_G[v]))\\
    &\le \frac{27}{64}\cdot 2^{n-1} + \frac{27}{64}\cdot 2^{n-2} + 3\cdot 2^{n-1-1-2}\\
    &=\frac{27}{64}\cdot 2^n \left( \frac{1}{2} + \frac{1}{4} + \frac{2}{9} \right)<\frac{27}{64}\cdot 2^n.
\end{align*}
Hence, we may assume that every component $G_i$, for $i \in [2]$, is a cycle. We, therefore, compute $\sigma_1(G)$ exhaustively as follows:
\begin{itemize}
    \item For $n=7$ we have
$
\sigma_1(C_3\cup C_4)=37<54=\frac{27}{64}\cdot 2^7.
$
\item For $n=8$ we have
$
\sigma_1(C_4\cup C_4)=56 < \sigma_1(C_3\cup C_5)=85<108=\frac{27}{64}\cdot 2^8.
$
\item For $n=9$ we have
$
\sigma_1(C_4\cup C_5)=130 < \sigma_1(C_3\cup C_6) = 165 < 216 = \frac{27}{64}\cdot 2^9.
$
\item For $n=10$ we have
$
\sigma_1(C_4\cup C_6) = 250 < \sigma_1(C_5\cup C_5)=300 =260<432=\frac{27}{64}\cdot 2^{10}.
$
\end{itemize}

This completes the proof of Theorem \ref{if-G-is-a-graph-of-order-n-6-then-sigma-1-at-most-max}.
\end{proof}


\newpage

\section{Appendix}

\subsection{$\sigma_1$ in graphs of small order}

In this appendix, we exhaustively compute $\sigma_1(G)$, where $G$ is any graph of order $n$, where $1\le n\le 6$. If $1\le n \le 4$, then $\sigma_1(G)$ is computed in Table \ref{sigma-1-in-graphs-of-small-order}.

\begin{table}[!h]
    \centering
\begin{tabular}{|c|l|c|}
        \hline
         $|V(G)|$ & $G$ & $\mathbf{\sigma_1(G)}$\\
         \hline
         $1$ & \begin{tikzpicture}[scale=0.5]
  \tikzstyle{vertexX}=[circle,draw, fill=white!90, minimum size=10pt, 
         scale=0.5, inner sep=0.2pt]
			\node (v_1) at (0,0) [vertexX, , label=left:] {};

		\end{tikzpicture} & $0$ \\
         \hline
         \multirow{2}{*}{$2$} & \begin{tikzpicture}[scale=0.5]
  \tikzstyle{vertexX}=[circle,draw, fill=white!90, minimum size=10pt, 
         scale=0.5, inner sep=0.2pt]
			\node (v_1) at (0,0) [vertexX, , label=left:] {};
                \node (v_2) at (1,0) [vertexX, , label=left:] {};
			
		\end{tikzpicture} & $0$\\\cline{2-3}
             & \begin{tikzpicture}[scale=0.5]
  \tikzstyle{vertexX}=[circle,draw, fill=white!90, minimum size=10pt, 
         scale=0.5, inner sep=0.2pt]
			\node (v_1) at (0,0) [vertexX, , label=left:] {};
                \node (v_2) at (1,0) [vertexX, , label=left:] {};
			
			\draw (v_1) -- (v_2);
		\end{tikzpicture} & $1$\\
\hline
    \end{tabular}
    \caption{$\sigma_1$ in graphs of order $n$, where $1\le n \le 4$.}
    \label{sigma-1-in-graphs-of-small-order}
\end{table}

\newpage
If $n = 5$, then $\sigma_1(G)$ is computed in Table \ref{sigma-1-in-graphs-of-order-5}.

\begin{table}[!h]
    \centering
    \begin{tabular}{|l|c|l|c|l|c|}
        \hline
         \textbf{A graph $G$} & $\mathbf{\sigma_1(G)}$ & \textbf{A graph $G$} & $\mathbf{\sigma_1(G)}$  & \textbf{A graph $G$} & $\mathbf{\sigma_1(G)}$ \\
\hline
    \end{tabular} 
    \caption{$\sigma_1$ in graphs of order $n=5$.}
    \label{sigma-1-in-graphs-of-order-5}
\end{table}

If $n = 6$, then $\sigma_1(G)$ is computed in Table \ref{sigma-1-in-graphs-of-order-6}.

\begin{table}[!h]
    \centering
    \begin{tabular}{|l|c|l|c|l|c|}
        \hline
         \textbf{A graph $G$} & $\mathbf{\sigma_1(G)}$ & \textbf{A graph $G$} & $\mathbf{\sigma_1(G)}$  & \textbf{A graph $G$} & $\mathbf{\sigma_1(G)}$ \\
\hline
\end{tabular}
\end{table}

\begin{table}[!h]
    \centering
    \begin{tabular}{|l|c|l|c|l|c|}
        \hline
         \textbf{A graph $G$} & $\mathbf{\sigma_1(G)}$ & \textbf{A graph $G$} & $\mathbf{\sigma_1(G)}$  & \textbf{A graph $G$} & $\mathbf{\sigma_1(G)}$ \\
\hline
\end{tabular}
\end{table}

\begin{table}[!h]
    \centering
    \begin{tabular}{|l|c|l|c|l|c|}
        \hline
         \textbf{A graph $G$} & $\mathbf{\sigma_1(G)}$ & \textbf{A graph $G$} & $\mathbf{\sigma_1(G)}$  & \textbf{A graph $G$} & $\mathbf{\sigma_1(G)}$ \\
\hline
\end{tabular}
\end{table}

\begin{table}[!h]
    \centering
    \begin{tabular}{|l|c|l|c|l|c|}
        \hline
         \textbf{A graph $G$} & $\mathbf{\sigma_1(G)}$ & \textbf{A graph $G$} & $\mathbf{\sigma_1(G)}$  & \textbf{A graph $G$} & $\mathbf{\sigma_1(G)}$ \\
\hline
    \end{tabular} 
    \caption{$\sigma_1$ in graphs of order $n=6$.}
    \label{sigma-1-in-graphs-of-order-6}
\end{table}



\newpage
\bibliographystyle{abbrv} 
\bibliography{references}

\end{document}